\documentclass[reqno]{amsart}
\usepackage{abstract}
\usepackage{geometry}
\usepackage{amsmath}
\usepackage[utf8]{inputenc}
\usepackage{hyperref}
\usepackage{amssymb}
\usepackage{amsfonts}
\usepackage{graphicx}
\usepackage{amsthm}
\usepackage{parskip}
\newtheorem{theorem}{Theorem}

\newtheorem{corollary}[theorem]{Corollary}

\newtheorem{definition}[theorem]{Definition}
\newtheorem{example}[theorem]{Example}

\newtheorem{lemma}[theorem]{Lemma}

\newtheorem{proposition}[theorem]{Proposition}
\newtheorem{remark}[theorem]{Remark}

\begin{document}
\title[Generalized Multiplicative Euler Phi]
{On The Generalized Multiplicative Euler Phi
Function }

\author{Mohammad El-Hindi and Therrar Kadri}

	\address{Mohammad Elhindi\newline  Department of Mathematics and Computer
		Science, Faculty of Science, Beirut Arab University, Beirut, Lebanon  } \email{mohammadyhindi98@gmail.com}
\address{Therrar Kadri \newline	Department of Pedagogy,
	Lebanese University.}\email{therrar@hotmail.com}

\maketitle

\begin{abstract}
The generalized group of units of the ring modulo $n$ was first introduced
by El-Kassar and Chehade, written as $U^k(Z_n)$. This allows us to formulate
a new generalization to the Euler phi function $\varphi(n)$, that represents the order of $%
U^k(Z_n)$ and it is denoted by $\varphi ^{k}(n).$ In this paper, we
introduce this newly defined function, where we compute its explicit form
and examine some of its properties similar to that of $\varphi (n)$. In
addition, we study some generalized equations involving $\varphi ^{k}(n)$
where complete solution is given for some equations by considering the
general case and others for some particular cases.
\end{abstract}

\section{Introduction}

In number theory, the Euler totient function $\varphi (n)$, also called
Euler's phi function, is defined for a given $n$ as the number of positive
integers are less than or equal to $n$ and relatively prime to $n$. The
Euler phi function has been extensively studied and there are many famous
equations involving it. One of the important equations on this function is
Lehmer's equation $k\varphi (n)=n\pm 1,$ which was first introduced in 1932
by Lehmer in \cite{6}. He conjectured that the only solutions for this
equation are prime integers. Since then, no one was able to provide a proof.
This problem is known to be Lehmer's totient problem and has received a
great attention.

The Euler phi function has been generalized through different approaches and
was studied in many domains. In 1983, J.T. Cross, in \cite{1}, extended the
definition of the Euler phi function to the domain of Gaussian integers,
where it represents the order of the multiplicative group of units of the
ring $\mathbb{Z} \lbrack i]/\langle\beta \rangle$, where $\beta$ is a non zero element in 
$\mathbb{Z}[i]$. Moreover, El-Kassar in \cite{2,3} generalized the
definition of the Euler phi function to any principal ideal domain $G$, and
denoted it by $\varphi \left( G(\beta )\right) $, that represents the order
of the group of units of $G/\langle\beta \rangle$, $\beta $ is a non zero element in $G$%
. The problem of finding the group of units of any commutative ring $R$
still an open problem. However, the problem is solved with $R= \mathbb{Z}
_{n}$. Also Cross in \cite{1} determined the structure of the group of units
of the factor ring of Gaussian integer modulo $\beta $. Smith and Gallian in 
\cite{8} solved the problem of decomposition of the group of units of the
finite ring $F[x]/\left\langle h(x)\right\rangle $, where $F$ is a finite
field and $h(x) $ is an irreducible polynomial in $F[x]$. A related problem
is that of determining the finite commutative rings with cyclic group of
units. This problem has been solved when $R= \mathbb{Z}_{n}$. It is proved
that $U(\mathbb{Z}_{n})$ is cyclic if and only if $n=2,4$ or $n$ of the form 
$p^{\alpha }$ or $2p^{\alpha }$ where $p$ is an odd prime and form all $%
\alpha >0$. Cross \cite{1} showed that the group of units of $\mathbb{Z}%
\lbrack i]/\left\langle \beta \right\rangle $ is cyclic if and only if $%
\beta =1+i, (1+i)^{2}, (1+i)^{3}, p, p(1+i), \pi ^{n}$ and $(1+i)\pi ^{n}$, where 
$p$ is a prime integer of the form $4k+3$ and $\pi $ is a Gaussian prime
such that $\overset{\_}{\pi }\pi $ is a prime integer of the form $4k+1$.
Moreover, El-Kassar in \cite{7} determined the quotient rings of polynomials
over a finite field having a cyclic group of units.

In 2006, a generalization for the group of units of any finite commutative
ring $R$ with identity, was introduced by El-Kassar and Chehade \cite{4}.
They proved that the group of units of a commutative ring $R$, $U(R)$,
supports a ring structure and this has made it possible to define the second
group of units of $R$ as, $U^{2}(R)=U(U(R))$. Extending this definition to
the $k$-th level, the $k$-th group of units are defined as, $%
U^{k}(R)=U(U^{k-1}(R))$. In addition, El-Kassar and Chehade considered the
problem of determining all finite commutative rings $R$ such that $U^{k}(R)$
is cyclic, and the problem of determining all finite commutative rings $R$
such that $U^{k}(R)$ is trivial. They solved both problems completely for $%
R= \mathbb{Z}_{n}$ and $k=2$. Later, El-Kadri and El-Kassar in \cite{5},
considered the problem for the case when $R=\mathbb{Z}_{n}$ and $k=3$ and
also provided a complete solution for these two problems.

On the other hand, the classical ElGamal public key encryption schemes
perhaps one of the most popular and widely used cryptosystems. The scheme is
best described in the setting of any finite cyclic group $G$. The security
of the generalized ElGamal encryption scheme is based on the intractability
of the discrete logarithmic problem in the group $G$. The group $G$ should
be carefully chosen so that the group operations in $G$ would be relatively
easy to apply for efficiency. So, one may consider ElGamal public key
cryptosystem using the cyclic group of units $U(R)$ or more generally $%
U^{k}(R)$. El-Kassar and Haraty \cite{4} extended the ElGamal cryptosystem
to the setting of quotient ring of polynomials having a cyclic group of
units $U( \mathbb{Z} _{2}[x]/<h(x)>)$, where $h(x)=$ $h_{1}(x)h_{2}(x)\ldots
h_{r}(x)$ is a product of irreducible polynomials whose degrees are pairwise
relatively prime. Later, Haraty al. \cite{9} gave another extension to the
ElGamal cryptosystem by employing the second group of units of $\mathbb{Z}%
_{n}$ and the second group of units of $\mathbb{Z}_{2}[x]/<h(x)>$, where $%
h(x)$ is irreducible.

In this paper, we determined the order of the $k$-th group of unit of the
ring $\mathbb{Z}_{n}$ defined as $\varphi ^{k}(n),$ where we determine an
explicit formula for it. Moreover, we solved some equations involving this
function that helped us to find more properties related to the $k$-th group
of units of $\mathbb{Z}_{n}.$

\section{Preliminaries}

In this section, we introduce the generalized group of
units, and some theorems related to this group. The next theorem was
introduced by Kadri and El-Hindi in \cite{12} that shows the decomposition
of of the generalized group of units. We note that through this paper "$\cong$" denotes ring isomorphism and "$\thickapprox$" denotes group isomorphism. 

\begin{theorem}
\label{thm} Let $p=2,$ then 
\begin{equation}
U^{k}(\mathbb{Z}_{p^{\alpha }})\approx \left\{ 
\begin{tabular}{ll}
$\left\{ 0\right\} $ & $\text{if }2k>\alpha $ \\ 
$\mathbb{Z}_{2}$ & $\text{if }2k=\alpha $ \\ 
$\mathbb{Z}_{2^{\alpha -2k+1}}$ & $\text{if\ }2k<\alpha$%
\end{tabular}
\right.  \label{5}
\end{equation}
\end{theorem}

if $p$ is an odd prime

\begin{equation*}
U^{k}(\mathbb{Z}_{p^{\alpha }})\approx \left\{ 
\begin{tabular}{ll}
$U^{k}(\mathbb{Z}_{p})\times U^{k-1}(\mathbb{Z}_{p})\times \ldots \times
U^{k-\alpha +1}(\mathbb{Z}_{p})$ & if$\text{ }k>\alpha $ \\ 
$U^{k}(\mathbb{Z}_{p})\times U^{k-1}(\mathbb{Z}_{p})\times \ldots \times U(%
\mathbb{Z}_{p})$ & if$\text{ }k=\alpha $ \\ 
$U^{k}(\mathbb{Z}_{p})\times U^{k-1}(\mathbb{Z}_{p})\times \ldots \times U(%
\mathbb{Z}_{p})\times \mathbb{Z}_{p^{\alpha -k}}$ & $\text{if\ }k<\alpha $%
\end{tabular}%
\ \right. 
\end{equation*}

\begin{example}
\label{example} We have 
\begin{equation*}
\mathbb{Z}_{1080000}=\mathbb{Z}_{2^{6}\times 3^{3}\times 5^{4}}\approx 
\mathbb{Z}_{2^{6}}\oplus \mathbb{Z}_{3^{3}}\oplus \mathbb{Z}_{5^{4}}
\end{equation*}%
then, 
\begin{align*}
& {{U}^{3}}\left( {{\mathbb{Z}}_{1080000}}\right) \approx {{U}^{3}}\left( {{%
\mathbb{Z}}_{{{2}^{6}}}}\right) \times {{U}^{3}}\left( {{\mathbb{Z}}_{{{3}%
^{3}}}}\right) \times {{U}^{3}}\left( {{\mathbb{Z}}_{{{5}^{4}}}}\right)  \\
& \,\,\,\,\,\,\,\,\,\,\,\,\,\,\,\,\,\,\,\,\,\,\,\,\,\,\,\,\,\,\approx {{%
\mathbb{Z}}_{2}}\times {{U}^{3}}\left( {{\mathbb{Z}}_{3}}\right) \times {{U}%
^{2}}\left( {{\mathbb{Z}}_{3}}\right) \times U\left( {{\mathbb{Z}}_{3}}%
\right) \times {{U}^{3}}\left( {{\mathbb{Z}}_{5}}\right) \times {{U}^{2}}%
\left( {{\mathbb{Z}}_{5}}\right) \times U\left( {{\mathbb{Z}}_{5}}\right)
\times {{\mathbb{Z}}_{5}} \\
& \,\,\,\,\,\,\,\,\,\,\,\,\,\,\,\,\,\,\,\,\,\,\,\,\,\,\,\,\,\,\approx {{%
\mathbb{Z}}_{2}}\times {{\mathbb{Z}}_{2}}\times {{\mathbb{Z}}_{2}}\times {{%
\mathbb{Z}}_{2}}\times {{\mathbb{Z}}_{4}}\times {{\mathbb{Z}}_{5}}
\end{align*}
\end{example}

The next theorem was introduced in \cite{11} which determines the cyclic case of $U^{2}(\mathbb{Z}_{n})$.

\begin{theorem}
\label{cycliccase} $U^{2}(\mathbb{Z}_{n})$ is cyclic if and only if $n$ is a
divisor of one the following:
\end{theorem}

\begin{enumerate}
\item $24p$ where $p$ is an odd prime either $p=5$ or $p=2q^{\alpha }+1$
where  $q$ is an odd prime and $\alpha>0$. 

\item $8\times 3^{\beta }$ for all $\beta $. 

\item $48$.
\end{enumerate}

\section{Generalized Euler Phi Function}

In this chapter, we introduce the generalization of the Euler phi function
denoted by $\varphi ^{k}(n)$ and derive the explicit expression 
of $\varphi ^{k}(n)$. We prove that $\varphi ^{k}(n)$ is a multiplicative
function, so to drive $\varphi ^{k}(n)$ it is enough to derive $\varphi
^{k}(p^{\alpha }),$ where $p$ is a prime number. Also, we solved some equations involving this function.

\subsection{Definition and properties}

In this section, we define the generalized Euler phi function, and we determine
its explicit formula.

\begin{definition}
\label{der} We define the order of the generalized group of units of the
ring $\mathbb{Z}_{n}$ to be the generalized Euler phi function denoted by $%
\varphi ^{k}(n)$ i.e. 
\begin{equation*}
\varphi ^{k}(n)=|U^{k}(\mathbb{Z}_{n})|.
\end{equation*}
\end{definition}

\begin{example}
Using Example \ref{example}, we have 
\begin{align*}
 \left| {{U}^{3}}\left( {{\mathbb{Z}}_{1080000}} \right) \right|&={{\varphi }^{3}}\left( 1080000 \right) \\
&=\left| {{\mathbb{Z}}_{2}}\times {{\mathbb{Z}}_{2}}\times {{\mathbb{Z}}_{2}}%
\times {{\mathbb{Z}}_{2}}\times {{\mathbb{Z}}_{4}}\times {{\mathbb{Z}}_{5}}
\right| \\
&=2\times 2\times 2\times 2\times 4\times 5=320. \\
\end{align*}
\end{example}

\begin{theorem}
\label{phk}

$\varphi ^{k}(n)$ is a multiplicative function.
\end{theorem}

\begin{proof}
Let $R=Z_{n}$ and $U^{k}(Z_{n})$ be its $k$-th group of units where the prime decomposition of $n$ is of the following form, $$n=p_{1}^{\alpha _{1}}p_{2}^{\alpha _{2}}\ldots p_{r}^{\alpha _{r}}$$ then,
$$U^{k}(Z_{n})\approx U^{k}(\mathbb{Z}_{p_{1}^{\alpha _{1}}})\oplus U^{k}(%
\mathbb{Z}_{p_{2}^{\alpha _{2}}})\oplus \ldots \oplus U^{k}(\mathbb{Z}_{p_{r}^{\alpha _{r}}})$$which leads to, $|U^{k}(\mathbb{Z}_{n})|$
$=|U^{k}(\mathbb{Z}_{p_{1}^{\alpha _{1}}})\oplus U^{k}(\mathbb{Z}_{p_{2}^{\alpha _{2}}})\oplus \ldots \oplus U^{k}(\mathbb{Z}_{p_{r}^{\alpha _{r}}})|$
$=|U^{k}(\mathbb{Z}_{p_{1}^{\alpha _{1}}})|\times |U^{k}(\mathbb{Z}_{p_{2}^{\alpha _{2}}})|\times \ldots \times |U^{k}(\mathbb{Z}_{p_{r}^{\alpha _{r}}})|.$ By Definition \ref{der} we get, 
$$\varphi ^{k}(n)=\prod\limits_{i=1}^{k}\varphi ^{k}(p_{i}^{\alpha _{i}})$$thus, $\varphi ^{k}(n)$ is a multiplicative function. 
\end{proof}

\begin{proposition}
If $p$ is an odd prime integer, then $\varphi ^{k}(p)=\varphi ^{k-1}(p-1)$.
\end{proposition}

\begin{proof}
Let $\mathbb{Z}_{n}$ the set of integers$\mod n$, and $U(\mathbb{Z}_{n})$ be its group of units, then by Theorem \ref{thm} $U(\mathbb{Z}_{p})\approx\mathbb{Z}_{p-1}$ i.e. $U^{k-1}(U(\mathbb{Z}_{p}))\approx U^{k-1}(\mathbb{Z}_{p-1})$ which is equivalent to 
$U^{k}(\mathbb{Z}_{p})$ $\approx U^{k-1}(\mathbb{Z}_{p-1})$, we find that ,$|U^{k}(\mathbb{Z}_{p})|$ $=|U^{k-1}(\mathbb{Z}_{p-1})|$, and therefore\\ $\varphi
^{k}(p)=\varphi ^{k-1}(p-1).$
\end{proof}

\begin{theorem}
\label{phi}Let $n$ be any positive integer where\ $n=p_{1}^{\alpha
_{1}}p_{2}^{\alpha _{2}}\ldots p_{r}^{\alpha _{r}}$ its prime decomposition, then $\varphi ^{k}(n)=\varphi
^{k}(p_{1}^{\alpha _{1}})\varphi ^{k}(p_{2}^{\alpha _{2}})\ldots \varphi
^{k}(p_{k}^{\alpha _{k}})$ where if$\ p_r=2$,

\begin{equation*}
\varphi ^{k}(p_r^{\alpha })=\left\{ 
\begin{array}{ll}
1 & \text{if }\alpha <2k \\ 
2^{\alpha -2k+1} & \text{if}\ \alpha \geq 2k%
\end{array}
\right. 
\end{equation*}
\end{theorem}

 and if $p_r$ is any odd prime integer,
\begin{equation*}
\varphi ^{k}(p_r^{\alpha })=\left\{ 
\begin{array}{ll}
\varphi ^{k-1}(p_r-1)\times \prod\limits_{i=k-a+1}^{k-1}\varphi ^{i}(p_r), & 
\text{if }\alpha <k \\ 
\varphi ^{k-1}(p_r-1)\times p_r^{\alpha -k}\times
\prod\limits_{i=1}^{k-1}\varphi ^{i}(p_r), & \text{if\ }\alpha \geq k%
\end{array}%
\right. 
\end{equation*}

\begin{proof}
From Proposition \ref{phk} we know that $\varphi ^{k}(n)$ is a 
multiplicative function and hence, $\varphi ^{k}(n)=\varphi ^{k}(p_{1}^{\alpha _{1}}p_{2}^{\alpha
_{2}}\ldots p_{k}^{\alpha _{k}})=\varphi ^{k}(p_{1}^{\alpha _{1}})\varphi 
^{k}(p_{2}^{\alpha _{2}})\ldots \varphi ^{k}(p_{k}^{\alpha _{k}}).$ Now for $
p_{r}=2$, by Theorem \ref{thm} we get  
\[
\left\vert U^{k}(\mathbb{Z}_{p_{r}^{\alpha }})\right\vert \approx \left\{  
\begin{tabular}{ll}
$\left\vert \left\{ 0\right\} \right\vert $ & $\text{if }2k>\alpha $ \\ 
$\left\vert

\mathbb{Z}

_{2}\right\vert $ & $\text{if }2k=\alpha $ \\ 
$\left\vert

\mathbb{Z}

_{2^{\alpha -2k+1}}\right\vert $ & $\text{if\ }2k<\alpha $%
\end{tabular}
\right.\]
by Definition \ref{der} we get, 
\[
\left\vert U^{k}(\mathbb{Z}_{2^{\alpha }})\right\vert =\varphi ^{k}(2^{\alpha})=\left\{  
\begin{array}{ll}
1 & \text{if }2k>\alpha \\ 
2 & \text{if }2k=\alpha \\ 
2^{\alpha -2k+1} & \text{if\ }2k<\alpha%
\end{array}
.\right.\]
that can be simplified to

\[
\varphi ^{k}(p^{\alpha })=\left\{  
\begin{array}{ll}
1 & \text{if }\alpha <2k \\ 
2^{\alpha -2k+1} & \text{if}\ \alpha \geq 2k%
\end{array}
\right.  
\]

Now, let $p$ be an odd prime. By Theorem \ref{thm} we get 
\[
\left\vert U^{k}(\mathbb{Z}_{p^{\alpha }})\right\vert \approx \left\{  
\begin{tabular}{ll}
$\left\vert U^{k}(

\mathbb{Z}

_{p})\right\vert \times \left\vert U^{k-1}(

\mathbb{Z}

_{p})\right\vert \times \ldots \times \left\vert U^{k-\alpha +1}(

\mathbb{Z}

_{p})\right\vert $ & if$\text{ }k>\alpha $ \\ 
$\left\vert U^{k}(

\mathbb{Z}

_{p})\right\vert \times \left\vert U^{k-1}(

\mathbb{Z}

_{p})\right\vert \times \ldots \times \left\vert U(

\mathbb{Z}

_{p})\right\vert $ & if$\text{ }k=\alpha $ \\ 
$\left\vert U^{k}(

\mathbb{Z}

_{p})\right\vert \times \left\vert U^{k-1}(

\mathbb{Z}

_{p})\right\vert \times \ldots \times \left\vert U(

\mathbb{Z}

_{p})\right\vert \times \left\vert

\mathbb{Z}

_{p^{\alpha -k}}\right\vert $ & $\text{if\ }k<\alpha $%
\end{tabular}
\right.\]

By Definition \ref{der}  
\[
\varphi ^{k}(p^{\alpha })=\left\{  
\begin{array}{ll}
\varphi ^{k}(p)\times \varphi ^{k-1}(p)\times \ldots \times \varphi
^{k-\alpha +1}(p) & \text{if }k>\alpha \\ 
\varphi ^{k}(p)\times \varphi ^{k-1}(p)\times \ldots \times \varphi (p) & 
\text{if }k=\alpha \\ 
\varphi ^{k}(p)\times \varphi ^{k-1}(p)\times \ldots \times \varphi
(p)\times p^{\alpha -k} & \text{if\ }k<\alpha%
\end{array}
\right.  
\]

By applying the previous lemma we get 
\[
\varphi ^{k}(p^{\alpha })=\left\{  
\begin{array}{ll}
\varphi ^{k-1}(p-1)\times \varphi ^{k-1}(p)\times \ldots \times \varphi
^{k-\alpha +1}(p) & \text{if }k>\alpha \\ 
\varphi ^{k-1}(p-1)\times \varphi ^{k-1}(p)\times \ldots \times \varphi (p)
& \text{if }k=\alpha \\ 
\varphi ^{k-1}(p-1)\times \varphi ^{k-1}(p)\times \ldots \times \varphi
(p)\times p^{\alpha -k} & \text{if\ }k<\alpha%
\end{array}
\right.  
\]

i.e. 
\[
\varphi ^{k}(p^{\alpha })=\left\{  
\begin{array}{ll}
\varphi ^{k-1}(p-1)\times \prod\limits_{i=k-a+1}^{k-1}\varphi ^{i}(p) & 
\text{if }k>\alpha \\ 
\varphi ^{k-1}(p-1)\times p^{\alpha -k}\times
\prod\limits_{i=1}^{k-1}\varphi ^{i}(p) & \text{if\ }k\leq \alpha%
\end{array}
\right.  
\]
\end{proof}

Here we illustrate an example applying Theorem \ref{phi}. Let $k=2$ and $n=7000$ then,

 \noindent $\varphi ^{2}(7000)=\varphi ^{2}(2^{3}\times 5^{3}\times 7)$

$\ \ \ \ \ \ \ \ \ \ \ \ \ =\varphi ^{2}(2^{3})\times \varphi
^{2}(5^{3})\times \varphi ^{2}(7)$ by Theorem \ref{phi}

$\ \ \ \ \ \ \ \ \ \ \ \ \ =1\times \varphi (4)\times \varphi (5)\times 5\times \varphi (6)=80.$

\section{Equations involving $\protect\varphi ^{k}(n)$ function}

In this section, we denote $\Phi ^{k}(n)$ to be $\Phi ^{k}(n)=\varphi
(\varphi (\ldots \varphi (n))),$ that is the application of the Euler phi
function on $n,$ $k$ times, see (\cite{13}). We consider the equation $%
\varphi ^{k}(n)=\Phi ^{k}(n),$ and we solved it completely for $k=2$ and $3.$
Moreover, we solved $\varphi ^{k}(n)=1$ completely for $k=2$ and $3$.

\begin{remark}
In general, $\varphi ^{k}(n)\neq \Phi ^{k}(n)$. Indeed, $\varphi ^{2}(7000)=80$ however, $\Phi ^{2}(7000)=640$. 
\end{remark}

\noindent In the next lemma, we determine a relation between $\varphi
(\prod\limits_{i=1}^{n}a_{i})$ and $\prod\limits_{i=1}^{n}\varphi (a_{i})$
where $n$ is a power of $2.$

\begin{lemma}
\label{powerof2}Let $\left\{ a_{1},a_{2},\ldots ,a_{n}\right\} ,$ be a set
of positive integers where $n=2^{r}$ and $r>0,$ then 

$\varphi (\prod\limits_{i=1}^{n}a_{i})=\prod\limits_{\substack{ 0<i<n  \\ %
2|i+1 }}\left[ \varphi (a_{i})\varphi (a_{i+1}) \displaystyle%
\frac{\gcd (a_{i},a_{i+1})}{\varphi \left( \gcd (a_{i},a_{i+1})\right) }%
\right] \\~~~~~~~~~~~~~
\times \prod\limits_{k=1}^{r-1}\left[ \prod\limits_{\substack{ 0<i<n 
\\ 2^{k+1}|i+1 }}\displaystyle\frac{\gcd
(\prod\limits_{j=0}^{2^{k}-1}a_{i+j},\prod
\limits_{j=2^{k}}^{2^{k+1}-1}a_{i+j})}{\varphi \left( \gcd
(\prod\limits_{j=0}^{2^{k}-1}a_{i+j},\prod
\limits_{j=2^{k}}^{2^{k+1}-1}a_{i+j})\right) }\right]$
\end{lemma}

\begin{proof}
The proof is done by induction on $n$.
\end{proof}
The next theorem plays an essential role in solving $\varphi ^{k}(n)=\Phi ^{k}(n)$.
\begin{theorem}
\label{pro}Let $\left\{ a_{1},a_{2},\ldots ,a_{n}\right\}$, be a set of
positive integers where $n=2^{r}m$ where $r\geq 0$ and $m\neq 1$ is an odd
integer then,

$\varphi (\prod\limits_{i=1}^{n}a_{i})=\prod\limits_{t=0}^{m-1}[\prod\limits 
_{\substack{ t\times 2^{r}<i<(1+t)\times 2^{r}  \\ 2|i+1}}\left[ \varphi
(a_{i}) \varphi (a_{i+1}) \displaystyle\frac{\gcd (a_{i},a_{i+1})%
}{\varphi \left( \gcd (a_{i},a_{i+1})\right) }\right] $

$\ \ \ \ \ \ \ \ \ \ \ \ \ \ \times \prod\limits_{k=1}^{r-1}\left[
\prod\limits_{\substack{ t\times 2^{r}<i<(1+t)\times 2^{r}  \\ 2^{k+1}|i+1}}%
\displaystyle\frac{\gcd
(\prod\limits_{j=0}^{2^{k}-1}a_{i+j},\prod%
\limits_{j=2^{k}}^{2^{k+1}-1}a_{i+j})}{\varphi \left( \gcd
(\prod\limits_{j=0}^{2^{k}-1}a_{i+j},\prod%
\limits_{j=2^{k}}^{2^{k+1}-1}a_{i+j})\right) }\right]$  $ \\~~~~~~~~~~~~~~~ \times \displaystyle%
\frac{\gcd (\prod\limits_{i=1}^{2^{r}}a_{i+t\times
2^{r}},\prod\limits_{i=2^{r}+1}^{2^{r}(m-t)}a_{i+t\times 2^{r}})}{\varphi
(\gcd (\prod\limits_{i=1}^{2^{r}}a_{i+t\times
2^{r}},\prod\limits_{i=2^{r}+1}^{2^{r}(m-t)}a_{i+t\times 2^{r}}))}].$
\end{theorem}
\begin{proof}
We have, \\$\varphi (\prod\limits_{i=1}^{n}a_{i})=\varphi 
(\prod\limits_{i=1}^{2^{r}}a_{i}\times 
\prod\limits_{i=2^{r}+1}^{2^{r}m}a_{i})$

$=\varphi (\prod\limits_{i=1}^{2^{r}}a_{i}) \varphi
(\prod\limits_{i=2^{r}+1}^{2^{r}m}a_{i}) \displaystyle\frac{\gcd
(\prod\limits_{i=1}^{2^{r}}a_{i},\prod\limits_{i=2^{r}+1}^{2^{r}m}a_{i})}{%
\varphi (\gcd
(\prod\limits_{i=1}^{2^{r}}a_{i},\prod\limits_{i=2^{r}+1}^{2^{r}m}a_{i}))}$

$=\varphi (\prod\limits_{i=1}^{2^{r}}a_{i}) \displaystyle\frac{\gcd
(\prod\limits_{i=1}^{2^{r}}a_{i},\prod\limits_{i=2^{r}+1}^{2^{r}m}a_{i})}{
\varphi (\gcd
(\prod\limits_{i=1}^{2^{r}}a_{i},\prod\limits_{i=2^{r}+1}^{2^{r}m}a_{i}))}
 \varphi (\prod\limits_{i=1}^{2^{r}}a_{i+2^{r}}\times 
\prod\limits_{i=2^{r}+1}^{2^{r}(m-1)}a_{i+2^{r}})$

$=\varphi (\prod\limits_{i=1}^{2^{r}}a_{i}) \varphi 
(\prod\limits_{i=1}^{2^{r}}a_{i+2^{r}}) \displaystyle\frac{\gcd
(\prod\limits_{i=1}^{2^{r}}a_{i},\prod\limits_{i=2^{r}+1}^{2^{r}m}a_{i})}{
\varphi (\gcd
(\prod\limits_{i=1}^{2^{r}}a_{i},\prod\limits_{i=2^{r}+1}^{2^{r}m}a_{i}))}
\displaystyle\frac{\gcd (\prod\limits_{i=1}^{2^{r}}a_{i+2^{r}},\prod
\limits_{i=2^{r}+1}^{2^{r}(m-1)}a_{i+2^{r}})}{\varphi (\gcd
(\prod\limits_{i=1}^{2^{r}}a_{i+2^{r}},\prod
\limits_{i=2^{r}+1}^{2^{r}(m-1)}a_{i+2^{r}})} \varphi 
(\prod\limits_{i=2^{r}+1}^{2^{r}(m-1)}a_{i+2^{r}})$

$=\varphi (\prod\limits_{i=1}^{2^{r}}a_{i}) \varphi 
(\prod\limits_{i=1}^{2^{r}}a_{i+2^{r}}) \displaystyle\frac{\gcd
(\prod\limits_{i=1}^{2^{r}}a_{i},\prod\limits_{i=2^{r}+1}^{2^{r}m}a_{i})}{
\varphi (\gcd
(\prod\limits_{i=1}^{2^{r}}a_{i},\prod\limits_{i=2^{r}+1}^{2^{r}m}a_{i}))}
 \displaystyle\frac{\gcd (\prod\limits_{i=1}^{2^{r}}a_{i+2^{r}},\prod
\limits_{i=2^{r}+1}^{2^{r}(m-1)}a_{i+2^{r}})}{\varphi (\gcd
(\prod\limits_{i=1}^{2^{r}}a_{i+2^{r}},\prod
\limits_{i=2^{r}+1}^{2^{r}(m-1)}a_{i+2^{r}})} \varphi 
(\prod\limits_{i=1}^{2^{r}(m-2)}a_{i+2\times 2^{r}})$

$=\varphi (\prod\limits_{i=1}^{2^{r}}a_{i}) \varphi 
(\prod\limits_{i=1}^{2^{r}}a_{i+2^{r}}) \displaystyle\frac{\gcd
(\prod\limits_{i=1}^{2^{r}}a_{i},\prod\limits_{i=2^{r}+1}^{2^{r}m}a_{i})}{
\varphi (\gcd
(\prod\limits_{i=1}^{2^{r}}a_{i},\prod\limits_{i=2^{r}+1}^{2^{r}m}a_{i}))}
\times \displaystyle\frac{\gcd (\prod\limits_{i=1}^{2^{r}}a_{i+2^{r}},\prod
\limits_{i=2^{r}+1}^{2^{r}(m-1)}a_{i+2^{r}})}{\varphi (\gcd
(\prod\limits_{i=1}^{2^{r}}a_{i+2^{r}},\prod
\limits_{i=2^{r}+1}^{2^{r}(m-1)}a_{i+2^{r}})}\times\\ \varphi 
(\prod\limits_{i=1}^{2^{r}}a_{i+2\times 2^{r}}\times 
\prod\limits_{i=2^{r}+1}^{2^{r}(m-2)}a_{i+2\times 2^{r}})$

$=\varphi (\prod\limits_{i=1}^{2^{r}}a_{i}) \varphi 
(\prod\limits_{i=1}^{2^{r}}a_{i+2^{r}}) \displaystyle\frac{\gcd
(\prod\limits_{i=1}^{2^{r}}a_{i},\prod\limits_{i=2^{r}+1}^{2^{r}m}a_{i})}{
\varphi (\gcd
(\prod\limits_{i=1}^{2^{r}}a_{i},\prod\limits_{i=2^{r}+1}^{2^{r}m}a_{i}))}
\times \displaystyle\frac{\gcd (\prod\limits_{i=1}^{2^{r}}a_{i+2^{r}},\prod
\limits_{i=2^{r}+1}^{2^{r}(m-1)}a_{i+2^{r}})}{\varphi (\gcd
(\prod\limits_{i=1}^{2^{r}}a_{i+2^{r}},\prod
\limits_{i=2^{r}+1}^{2^{r}(m-1)}a_{i+2^{r}})}\times \\
\varphi 
(\prod\limits_{i=1}^{2^{r}}a_{i+2\times 2^{r}}) \varphi 
(\prod\limits_{i=2^{r}+1}^{2^{r}(m-2)}a_{i+2\times 2^{r}})\times \displaystyle\frac{\gcd
(\prod\limits_{i=1}^{2^{r}}a_{i+2\times
2^{r}},\prod\limits_{i=2^{r}+1}^{2^{r}(m-2)}a_{i+2\times 2^{r}})}{\varphi
(\gcd (\prod\limits_{i=1}^{2^{r}}a_{i+2\times
2^{r}},\prod\limits_{i=2^{r}+1}^{2^{r}(m-2)}a_{i+2\times 2^{r}}))}$

$=\varphi (\prod\limits_{i=1}^{2^{r}}a_{i}) \varphi 
(\prod\limits_{i=1}^{2^{r}}a_{i+2^{r}}) \varphi 
(\prod\limits_{i=1}^{2^{r}}a_{i+2\times 2^{r}})\times \displaystyle\frac{\gcd
(\prod\limits_{i=1}^{2^{r}}a_{i},\prod\limits_{i=2^{r}+1}^{2^{r}m}a_{i})}{
\varphi (\gcd
(\prod\limits_{i=1}^{2^{r}}a_{i},\prod\limits_{i=2^{r}+1}^{2^{r}m}a_{i}))}
\times\\ \displaystyle\frac{\gcd (\prod\limits_{i=1}^{2^{r}}a_{i+2^{r}},\prod
\limits_{i=2^{r}+1}^{2^{r}(m-1)}a_{i+2^{r}})}{\varphi (\gcd
(\prod\limits_{i=1}^{2^{r}}a_{i+2^{r}},\prod
\limits_{i=2^{r}+1}^{2^{r}(m-1)}a_{i+2^{r}})}\times \displaystyle\frac{\gcd
(\prod\limits_{i=1}^{2^{r}}a_{i+2\times
2^{r}},\prod\limits_{i=2^{r}+1}^{2^{r}(m-2)}a_{i+2\times 2^{r}})}{\varphi
(\gcd (\prod\limits_{i=1}^{2^{r}}a_{i+2\times
2^{r}},\prod\limits_{i=2^{r}+1}^{2^{r}(m-2)}a_{i+2\times 2^{r}}))}
\varphi (\prod\limits_{i=1}^{2^{r}(m-3)}a_{i+3\times 2^{r}})$
\\\\\\
by continuing in the same manner we get,
\\\\\\
$\varphi (\prod\limits_{i=1}^{n}a_{i})=\prod\limits_{t=0}^{m-1}\left[
\varphi (\prod\limits_{i=1}^{2^{r}}a_{i+t\times 2^{r}})\times \displaystyle\frac{\gcd
(\prod\limits_{i=1}^{2^{r}}a_{i+t\times
2^{r}},\prod\limits_{i=2^{r}+1}^{2^{r}(m-t)}a_{i+t\times 2^{r}})}{\varphi
(\gcd (\prod\limits_{i=1}^{2^{r}}a_{i+t\times
2^{r}},\prod\limits_{i=2^{r}+1}^{2^{r}(m-t)}a_{i+t\times 2^{r}}))}\right] $

$\ \ \ \ \ \ \ \ \ \ \ =\prod\limits_{t=0}^{m-1}\left[ \varphi
(\prod\limits_{i=t\times 2^{r}+1}^{(1+t)\times 2^{r}}a_{i})\times \displaystyle\frac{\gcd
(\prod\limits_{i=1}^{2^{r}}a_{i+t\times
2^{r}},\prod\limits_{i=2^{r}+1}^{2^{r}(m-t)}a_{i+t\times 2^{r}})}{\varphi
(\gcd (\prod\limits_{i=1}^{2^{r}}a_{i+t\times
2^{r}},\prod\limits_{i=2^{r}+1}^{2^{r}(m-t)}a_{i+t\times 2^{r}}))}\right] $
\\\\\
and by Lemma \ref{powerof2} we get,
\\\\\\
$=\prod\limits_{t=0}^{m-1}[\prod\limits_{\substack{ t\times
2^{r}<i<(1+t)\times 2^{r}  \\ 2|i+1}}\left[ \varphi (a_{i})\varphi 
(a_{i+1})\displaystyle\frac{\gcd (a_{i},a_{i+1})}{\varphi \left( \gcd
(a_{i},a_{i+1})\right) }\right] $

$\times \prod\limits_{k=1}^{r-1}\left[ \prod\limits_{\substack{ t\times
2^{r}<i<(1+t)\times 2^{r}  \\ 2^{k+1}|i+1}}\displaystyle\frac{\gcd
(\prod\limits_{j=0}^{2^{k}-1}a_{i+j},\prod
\limits_{j=2^{k}}^{2^{k+1}-1}a_{i+j})}{\varphi \left( \gcd
(\prod\limits_{j=0}^{2^{k}-1}a_{i+j},\prod
\limits_{j=2^{k}}^{2^{k+1}-1}a_{i+j})\right) }\right]  \displaystyle\frac{\gcd
(\prod\limits_{i=1}^{2^{r}}a_{i+t\times
2^{r}},\prod\limits_{i=2^{r}+1}^{2^{r}(m-t)}a_{i+t\times 2^{r}})}{\varphi
(\gcd (\prod\limits_{i=1}^{2^{r}}a_{i+t\times
2^{r}},\prod\limits_{i=2^{r}+1}^{2^{r}(m-t)}a_{i+t\times 2^{r}}))}]$ 
\end{proof}

\begin{example}
For $n=6,6=2\times 3,$ then $m=3$ and $r=1$. Let $s=\{5,8,9,13,18,22\}$, then by applying the previous
theorem we get\\
$\varphi (5\times 8\times 9\times 13\times 18\times 22)$
$\\=\varphi (5) \varphi (8) \varphi (9) \varphi (13)
\varphi (18) \varphi (22)\times \displaystyle\frac{\gcd (5,8)}{\varphi
\left( \gcd (5,8)\right) }\times \displaystyle\frac{\gcd (9,13)}{\varphi
\left( \gcd (9,13)\right) }\times \displaystyle\frac{\gcd (18,22)}{\varphi
\left( \gcd (18,22)\right) }\\\times \displaystyle\frac{\gcd (5\times
8,9\times 13\times 18\times 22)}{\varphi (\gcd (5\times 8,9\times 13\times
18\times 22))}\times \displaystyle\frac{\gcd (9\times 13,18\times 22)}{%
\varphi (\gcd (9\times 13,18\times 22))}=414720.$
\end{example}

\subsection{The equation $\protect\varphi ^{2}(n)=\Phi ^{2}(n).$}

In this section, we solve $\varphi ^{2}(n)=\Phi ^{2}(n)$ and we find that it
has a direct relation with the cyclic case of $U\left( 
\mathbb{Z}
_{n}\right) $. The following lemmas play an essential role to build the
proof of the main results in this section.

\begin{lemma}
	Let $\varphi ^{2}(n)=\Phi ^{2}(n).$ If $n$ is a power of $2$ i.e. $%
	n=2^{\alpha }$, then $\alpha =1$ or $2.$
\end{lemma}

\begin{proof}
	Let $n=2^{\alpha }$, where $\alpha <4.$ If $\varphi ^{2}(2^{\alpha })=\Phi
	^{2}(2^{\alpha })$ then by Theorem \ref{phi}, then $1=2^{\alpha -2}$ then $\alpha =1\ \text{or }2$. Now for $\alpha \geq 4,$ $\varphi ^{2}(2^{\alpha })=\Phi ^{2}(2^{\alpha })$
	then, $2^{\alpha -3}=2^{\alpha -2}$ which is impossible, and this concludes the
	result.
\end{proof}

\begin{lemma}\label{lem14}
	Let $p$ be an odd prime, then $n=p^{\alpha }$ where $\alpha >0$ satisfies $%
	\varphi ^{2}(n)=\Phi ^{2}(n).$
\end{lemma}

\begin{proof}
	Let $n=p^{\alpha }$ where $p$ is an odd prime and $\alpha =1$. If $\varphi
	^{2}(p)=\Phi ^{2}(p)$ then, $\varphi (p-1)=\varphi (p-1)$ which true for
	every odd prime $p$. Now, let $n=p^{\alpha }$ where $\alpha \geq 2$. If $%
	\varphi ^{2}(p^{\alpha })=\Phi ^{2}(p^{\alpha })$ then,%
	\[
	\varphi (p-1)(p-1)p^{\alpha -2}=\varphi (p-1)(p-1)p^{\alpha -2}.
	\]
	
	Then, $\varphi ^{2}(p^{\alpha })=\Phi ^{2}(p^{\alpha })$ for any prime $p$
	and $\alpha >0.$
\end{proof}

\begin{lemma}\label{pr}
	Let $n=2^{\alpha }p^{\beta }$ where $\alpha $,$\beta >0$ and $p$ is an odd
	prime. $\varphi ^{2}(n)=\Phi ^{2}(n)$ if $\alpha =0$ or $1$
	and $\beta >0.$
\end{lemma}

\begin{proof}
	\label{rem}The case when $\alpha =0$ is discussed in the Lemma \ref{lem14}. Let 
	$n=2^{\alpha }p^{\beta }$ where $p$ is an odd prime, $\alpha <4,$ and $\beta
	=1$. If $\varphi ^{2}(2^{\alpha }p)=\Phi ^{2}(2^{\alpha }p)$ then,%
	\[
	\varphi(p-1)=\varphi(2^{\alpha-1}(p-1))
	\]%
	i.e.,
	
	\[
	\varphi (p-1)=\varphi (2^{\alpha -1})\times \varphi (p-1)\times \frac{\gcd
		(2^{\alpha -1},p-1)}{\varphi (\gcd (2^{\alpha -1},p-1))}
	\]
	however, $\gcd (2^{\alpha -1},p-1)$ must have the form of $2^{i}$ for some $i
	$ satisfying $0\leq i\leq \alpha -1,$then
	
	\[
	\varphi (p-1)=\varphi (2^{\alpha -1})\varphi (p-1)\times \frac{2^{i}}{2^{i-1}%
	}
	\]
	then, $\varphi (p-1)=2^{\alpha -2}\varphi (p-1)\times 2$ we conclude that $%
	1=2^{\alpha -1}$ i.e. $\alpha =1$. Now, let $n=2^{\alpha }p^{\beta }$ and $%
	\alpha \geq 4$. If $\varphi ^{2}(2^{\alpha }p)=\Phi ^{2}(2^{\alpha }p)$ we
	get,
	
	\[
	2^{\alpha -3}\times \varphi (p-1)=2^{\alpha -2}\times \varphi (p-1)\times 2
	\]
	which is impossible, then there is no solution of the form $2^{\alpha
	}p^{\beta }$ with $\alpha \geq 4.$
	
	Now, let $n=2^{\alpha }p^{\beta }$ where $p$ is an odd prime and $\alpha <4$
	and $\beta >1$. If $\varphi ^{2}(2^{\alpha }p^{\beta })=\Phi ^{2}(2^{\alpha
	}p^{\beta })$ then%
	\[
	\varphi (p-1)\times (p-1)\times p^{\beta -2}=\varphi (2^{\alpha
		-1}(p-1)p^{\beta -1})
	\]
	
	by using Theorem \ref{pro} for $n=3$ we get,%
	\[
	\varphi (p-1)(p-1)p^{\beta -2}=\varphi (2^{\alpha -1})\varphi (p-1)\varphi
	(p^{\beta -1})\frac{\gcd (2^{\alpha -1},(p-1)p^{\beta -1})}{\varphi (\gcd
		(2^{\alpha -1},(p-1)p^{\beta -1}))}\times \frac{\gcd (p-1,p^{\beta -1})}{%
		\varphi (\gcd (p-1,p^{\beta -1}))}
	\]
	
	we get,%
	\[
	\varphi (p-1)(p-1)p^{\beta -2}=2^{\alpha -2}\varphi (p-1)(p-1)p^{\beta
		-2}\times 2
	\]
	then, $2^{\alpha -1}=1$ i.e., $\alpha =1.$ Now, let $n=2^{\alpha }p^{\beta }$ where $p$ is an odd prime and $\alpha \geq 4$ and $\beta >1,$ by following
	the same previous procedure we get, 
	\[
	2^{\alpha -1}=2^{\alpha -3}
	\]%
	which is impossible then, $n=2^{\alpha }p^{\beta }$where $p$ is an odd prime
	is not a solution for any $\alpha \geq 4$ and $\beta >1$. We conclude that $%
	n=2^{\alpha }p^{\beta }$ where $\alpha $,$\beta >0$ and $p$ is an odd prime
	is a solution $\varphi ^{2}(n)=\Phi ^{2}(n)$ if $\alpha =0$ or $1$ and $%
	\beta >0$
\end{proof}

\begin{lemma}
	$\varphi ^{2}(n)=\Phi ^{2}(n)$ has no solution if $n$ has $2$ or more
	distinct odd prime factors.
\end{lemma}

\begin{proof}
	Let $n=p^{\alpha }q^{\beta }$ where $p$ and $q$ are distinct odd primes and $%
	\alpha =\beta =1.$ If $\varphi ^{2}(pq)=\Phi ^{2}(pq)$ then,  
	\[
	\varphi (p-1)\times \varphi (q-1)=\varphi (p-1)\times (q-1)\times \frac{\gcd
		(p-1,q-1)}{\varphi (\gcd (p-1,q-1))}
	\]
After some simplifications we get,
	\[
	\gcd (p-1,q-1)=\varphi (\gcd (p-1,q-1))
	\]
	then $\gcd (p-1,q-1)=1$ which is impossible since $p-1$ and $q-1$ are both
	even numbers. Therefore, there is no solution of the form $n=pq.$ Now,
	suppose that one of the primes has power greater than $1,$ i.e. $%
	n=p^{\alpha }q$ where $p,q$ are distinct odd primes and $\alpha >1$. If $%
	\varphi ^{2}(p^{\alpha }q)=\Phi ^{2}(p^{\alpha }q)$ then by Theorem \ref{pro} for $%
	n=3$ we get, \\\\
	$~~~~\varphi (p-1)(p-1)p^{\alpha -2}\varphi (q-1)
	\\=\varphi (p-1)\varphi (p^{\alpha
		-1})\varphi (q-1)\frac{\gcd ((p-1)p^{\alpha -1},q-1)}{\varphi (\gcd
		((p-1)p^{\alpha -1},q-1))}
	\times \frac{\gcd (p-1,p^{\alpha -1})}{\varphi
		(\gcd (p-1,p^{\alpha -1}))}$\\
	\\i.e., 
	\[
	\gcd ((p-1)\times p^{\alpha -1},q-1)=\varphi (\gcd ((p-1)\times p^{\alpha
		-1},q-1))
	\]%
	then, $\gcd ((p-1)\times p^{\alpha -1},q-1)=1$ which is impossible since $p-1
	$ and $q-1$ are evens then $(p-1)\times p^{\alpha -1}$ and $q-1$ are also
	evens then, $2|\gcd ((p-1)\times p^{\alpha -1},q-1).$ Finally, the same
	procedure is done for;.' $n=p^{\alpha }q^{\beta }$ where $p,q$ are distinct odd
	primes and $\alpha >1,\beta >1.$ On the same manner, by using the same steps
	used above, $n=p_{1}^{\alpha _{1}}p_{2}^{\alpha _{2}}\ldots p_{k}^{\alpha
		_{k}}$ does not satisfy $\varphi ^{2}(n)=\Phi ^{2}(n)$.
\end{proof}

In the next theorem, we summarize all previous lemmas to generate our main result given in the following theorem. 

\begin{theorem}
	\label{eq1}$\varphi ^{2}(n)=\Phi ^{2}(n)$ if and only if $n=2,4,p^{\alpha }$
	or $2p^{\alpha }$ where $p$ is an odd prime.
\end{theorem}

We note that from Theorem \ref{eq1} that $\varphi ^{2}(n)=\Phi ^{2}(n)$ if
and only if $U\left( 
\mathbb{Z}
_{n}\right) $ is cyclic.\ The next corollary is very usefull in the next
section.

\begin{corollary}
	\label{gr}For all $n>0,$ $\varphi ^{2}(n)\leq \Phi ^{2}(n).$
\end{corollary}

\begin{proof}
	If $n$ is a sqaure-free integer then i.e. $n=p_{1}p_{2}\ldots p_{k}$ then, $%
	\varphi ^{2}(n)=\varphi ^{2}(p_{1}p_{2}\ldots p_{k})=\varphi
	^{2}(p_{1})\varphi ^{2}(p_{2})\ldots \varphi ^{2}(p_{k})=\varphi
	(p_{1}-1)\varphi (p_{2}-1)\ldots \varphi (p_{k}-1)$.
	
	On the other hand, \\$\Phi ^{2}(n)=\Phi ^{2}(p_{1}p_{2}\ldots p_{k})=\varphi
	(\varphi (p_{1})\varphi (p_{2})\ldots \varphi (p_{k}))=\varphi
	((p_{1}-1)(p_{2}-1)\ldots (p_{k}-1))$.
	
	Now let $k=2^{r}m$ then, by Theorem \ref{pro},
	
	$\varphi ((p_{1}-1)(p_{2}-1)\ldots
	(p_{k}-1))=\prod\limits_{t=0}^{m-1}[\prod\limits_{\substack{ t\times
			2^{r}<i<(1+t)\times 2^{r} \\ 2|i+1}}\left[ \varphi (p_{i}-1)\times \varphi
	(p_{i+1}-1)\times \frac{\gcd (p_{i}-1,p_{i+1}-1)}{\varphi \left( \gcd
		(p_{i}-1,p_{i+1}-1)\right) }\right] $
	
	$\times \prod\limits_{k=1}^{r-1}\left[ \prod\limits_{\substack{ t\times
			2^{r}<i<(1+t)\times 2^{r} \\ 2^{k+1}|i+1}}\frac{\gcd
		(\prod\limits_{j=0}^{2^{k}-1}p_{i+j}-1,\prod%
		\limits_{j=2^{k}}^{2^{k+1}-1}(p_{i+j}-1))}{\varphi \left( \gcd
		(\prod\limits_{j=0}^{2^{k}-1}p_{i+j}-1,\prod%
		\limits_{j=2^{k}}^{2^{k+1}-1}(p_{i+j}-1))\right) }\right] \times \frac{\gcd
		(\prod\limits_{i=1}^{2^{r}}p_{i-1+t\times
			2^{r}},\prod\limits_{i=2^{r}+1}^{2^{r}(m-t)}p_{i-1+t\times 2^{r}})}{\gcd
		(\prod\limits_{i=1}^{2^{r}}p_{i-1+t\times
			2^{r}},\prod\limits_{i=2^{r}+1}^{2^{r}(m-t)}p_{i-1+t\times 2^{r}})}]$
	
	$=\prod\limits_{i=1}^{k}\varphi (p_{i}-1)\times
	\prod\limits_{t=0}^{m-1}[\prod\limits_{\substack{ t\times
			2^{r}<i<(1+t)\times 2^{r} \\ 2|i+1}}\left( \frac{\gcd (p_{i}-1,p_{i+1}-1)}{%
		\varphi \left( \gcd (p_{i}-1,p_{i+1}-1)\right) }\right) $
	
	$\times \prod\limits_{k=1}^{r-1}\left[ \prod\limits_{\substack{ t\times
			2^{r}<i<(1+t)\times 2^{r} \\ 2^{k+1}|i+1}}\frac{\gcd
		(\prod\limits_{j=0}^{2^{k}-1}p_{i+j}-1,\prod%
		\limits_{j=2^{k}}^{2^{k+1}-1}(p_{i+j}-1))}{\varphi \left( \gcd
		(\prod\limits_{j=0}^{2^{k}-1}p_{i+j}-1,\prod%
		\limits_{j=2^{k}}^{2^{k+1}-1}(p_{i+j}-1))\right) }\right] \times \frac{\gcd
		(\prod\limits_{i=1}^{2^{r}}p_{i-1+t\times
			2^{r}},\prod\limits_{i=2^{r}+1}^{2^{r}(m-t)}p_{i-1+t\times 2^{r}})}{\varphi
		(\gcd (\prod\limits_{i=1}^{2^{r}}p_{i-1+t\times
			2^{r}},\prod\limits_{i=2^{r}+1}^{2^{r}(m-t)}p_{i-1+t\times 2^{r}}))}%
	=\prod\limits_{i=1}^{k}\varphi (p_{i}-1)\times c$ where $c\geq 1,$ since for
	all $n>0,$ $\frac{n}{\varphi (n)}\geq 1,$ then $\varphi ^{2}(n)\leq \Phi
	^{2}(n)$.
	
	Now if $n$ is not a square-free, then there exist at least one prime factor
	has a power greater then one, say $p_{i},$ i.e. $n=$ $p_{1}p_{2}\ldots
	p_{i}^{\alpha }\ldots p_{k}$ where $\alpha >1.$ We have,
	
	\[
	\varphi ^{2}(n)=(p_{i}-1)p_{i}^{\alpha -2}\varphi (p_{1}-1)\varphi
	(p_{2}-1)\ldots \varphi (p_{i}-1)\ldots \varphi (p_{k}-1)
	\]
	
	On the other hand,
	
	\[
	\Phi ^{2}(n)=\varphi ((p_{1}-1)(p_{2}-1)\ldots p_{i}^{\alpha
		-1}(p_{i}-1)\ldots \varphi (p_{k}))
	\]
	
	and since $p_{i}^{\alpha -1}$ is odd and $p_{j}-1$ is even for all $j$ then, \\\\
	$~~~~\varphi ((p_{1}-1)(p_{2}-1)\ldots p_{i}^{\alpha -1}(p_{i}-1)\ldots \varphi
	(p_{k}))\\=\varphi (p_{i}^{\alpha -1})\varphi ((p_{1}-1)(p_{2}-1)\ldots
	(p_{i}-1)\ldots \varphi (p_{k}))$
	
	$=(p_{i}-1)p_{i}^{\alpha -2}\varphi (p_{1}-1)\varphi (p_{2}-1)\ldots \varphi
	(p_{i}-1)\ldots \varphi (p_{k}-1)\times c$ but $c\geq 1$ therefore $\varphi
	^{2}(n)\leq \Phi ^{2}(n)$ for all integers $n.$
\end{proof}

\subsection{The equation $\protect\varphi ^{3}(n)=\Phi ^{3}(n).$}

In this section, we solve the equation 

\begin{equation}
\varphi ^{3}(n)=\Phi ^{3}(n)\text{.}  \label{3}
\end{equation}

We start by some lemmas to obtain a general solution of Eq. \ref{3}.

\begin{lemma}
	If $\varphi ^{3}(n)=\Phi ^{3}(n)$ and $n=2^{\alpha }$, where $\alpha >0$ then 
	$n=2,4$ or $8.$
\end{lemma}

\begin{proof}
	Let $n=2^{\alpha }$ where $\alpha <6.$ If $\varphi ^{3}(2^{\alpha })=\Phi
	^{3}(2^{\alpha })$ then, $1=2^{\alpha -3}$ i.e. $\alpha =1,2$ or $3.$ Now let $n=2^{\alpha }$, where $\alpha \geq 6$. If $\varphi ^{3}(2^{\alpha
	})=\Phi ^{3}(2^{\alpha })$ then, $2^{\alpha -3}=2^{\alpha -5}$ which is
	impossible. Then the only solution in this case are $2,4$ and $8.$
\end{proof}

\begin{lemma}
	\label{prime}If $\varphi ^{3}(n)=\Phi ^{3}(n)$ where $n$ is an odd prime
	then $n=3,5$ or $n$ is a prime of the form $q^{\alpha }+1,$ $2q^{\alpha }+1$
	where $q$ is an odd prime.
\end{lemma}

\begin{proof}
	let $\varphi ^{3}(n)=\Phi ^{3}(n)$ and $n=p$ where $p$ is an odd prime. If $%
	\varphi ^{3}(p)=\Phi ^{3}(p)$ then,  
	\[
	\varphi ^{2}(p-1)=\Phi ^{2}(p-1)
	\]
	then by Theorem \ref{eq1}, $p-1=2,4,q^{\alpha },$ or $2q^{\alpha }$ i.e. $%
	n=3,5$ or $n$ is a prime of the form $q^{\alpha }+1,$ $2q^{\alpha }+1$ where 
	$q$ is an odd prime.
\end{proof}

\begin{lemma}
	If $\varphi ^{3}(n)=\Phi ^{3}(n)$ and $n=p^{2}$ where $p$ is an odd prime
	then $p=3.$
\end{lemma}

\begin{proof}
	let $\varphi ^{3}(n)=\Phi ^{3}(n)$ and $n=p^{2}$ where $p$ is an odd prime.
	If $\varphi ^{3}(p^{2})=\Phi ^{3}(p^{2})$ then,%
	\[
	\varphi ^{2}((p-1)p)=\Phi ^{2}(p(p-1))
	\]
	then by Theorem \ref{eq1} $p(p-1)=2,4,q^{\alpha }$ or $2q^{\alpha }$. If $p(p-1)=2$, then $p=2$ which is rejected since $p$ is odd. Moreover, there is no such prime $p$ satisfying $p(p-1)=4$. In addition the case, $p(p-1)=q^{\alpha }$ \ is rejected since $q^{\alpha }$ is an odd number while $p(p-1)$ is even. Now, we discuss the final case  $%
	p(p-1)=2q^{\alpha }$. Which gives $p-1=2$, $q=p$, and $\alpha =1,$ i.e. $n=9$%
	.
\end{proof}

\begin{lemma}
	If $\varphi ^{3}(n)=\Phi ^{3}(n)$, $n=p^{\alpha }$ and $p$ is an odd prime, and $\alpha >2$ then $p=3.$
\end{lemma}

\begin{proof}
	Let $\varphi ^{3}(n)=\Phi ^{3}(n)$, $n=p^{\alpha }$ where $p$ is an odd
	prime, and $\alpha >2$ then $\varphi ^{3}(p^{\alpha })=\Phi ^{3}(p^{\alpha })
	$ gives,%
	\begin{equation}
	\varphi ^{2}((p-1)p)\times (p-1)\times p^{\alpha -3}=\Phi ^{2}(p^{\alpha
		-1}(p-1)).  \label{11}
	\end{equation}
	We have,%
	\begin{eqnarray}
	&&\Phi ^{2}(p^{\alpha -1}(p-1))  \nonumber \\
	&=&\varphi (p^{\alpha -2}\times (p-1)\times \varphi (p-1))  \nonumber \\
	&=&(p-1)\times p^{\alpha -3}\times \Phi ^{2}((p-1)^{2})  \label{33}
	\end{eqnarray}
	
	combining equations $\left( \ref{11}\right) $ and $\left( \ref{33}\right) $ we get, 
	
	\[
	\varphi ^{2}((p-1)p)=\varphi ((p-1)\times \varphi (p-1))
	\]
	
	i.e.
	
	\[
	\varphi ^{2}(p-1)=\Phi ^{2}(p-1)\times \frac{\gcd [p-1,\varphi (p-1)]}{%
		\varphi (\gcd [p-1,\varphi (p-1)])}
	\]
	however, $\frac{\gcd [p-1,\varphi (p-1)]}{\varphi (\gcd [p-1,\varphi (p-1)])}%
	\geq 1$ then, $\varphi ^{2}(p-1)\geq \Phi ^{2}(p-1)$, from Corollary $\left( \ref{gr}%
	\right) $ we have 
	\[
	\varphi ^{2}(p-1)\leq \Phi ^{2}(p-1)
	\]%
	then,
	
	\[
	\varphi ^{2}(p-1)=\Phi ^{2}(p-1)
	\]
	which leads us to, $\frac{\gcd [p-1,\varphi (p-1)]}{\varphi (\gcd
		[p-1,\varphi (p-1)])}=1$ i.e. $\gcd [p-1,\varphi (p-1)]=1$ only if $\varphi
	(p-1)=1$ i.e. $p-1=2$ therefore, $p=3.$
\end{proof}

\begin{lemma}
	Let $n=2^{\alpha }p$ where $\alpha <6$ and $p$ is an odd prime and $%
	\varphi ^{3}(n)=\Phi ^{3}(n)$, then $\alpha =1$ and $p=3,5$ or $n$ is a
	prime of the form $q^{\alpha }+1,$ $2q^{\alpha }+1$ where $q$ is an odd
	prime.
\end{lemma}

\begin{proof}
	\label{rem1}Let $n=2^{\alpha }p$ where $\alpha <6$ and $p$ is an odd prime,
	then $\varphi ^{3}(n)=\Phi ^{3}(n)$ implies,
	
	\[
	\varphi ^{3}(p)=\varphi (2^{\alpha -2}\varphi (p-1)\frac{\gcd (2^{\alpha
			-1},p-1)}{\varphi (\gcd (2^{\alpha -1},p-1))})
	\]
	
	following the same procedure of the Proof of Lemma \ref{rem} we get,%
	\[
	\varphi ^{2}(p-1)=2^{\alpha -1}\Phi ^{2}(p-1)
	\]
	we conclude that,%
	\[
	\varphi ^{2}(p-1)\geq \Phi ^{2}(p-1)
	\]
	
	however by Corollary \ref{gr}, $\varphi ^{2}(p-1)\leq \Phi ^{2}(p-1)$ then, $%
	\varphi ^{2}(p-1)=\Phi ^{2}(p-1)$ and $2^{\alpha -1}=1$. Finally, by Theorem %
	\ref{prime}, we conclude the following, $p=3,5$ or $p$ is a prime of the
	form $q^{\beta }+1,$ $2q^{\beta }+1$ where $q$ is an odd prime, and $\alpha
	=1.$
\end{proof}

\begin{lemma}
	Let $n=2^{\alpha }p$ where $\alpha \geq 6$ and $p$ is an odd prime, then $n$
	does not satisfy $\varphi ^{3}(n)=\Phi ^{3}(n)$.
\end{lemma}

\begin{proof}
	Let $n=2^{\alpha }p$ where $\alpha \geq 6$ and $p$ is an odd prime, suppose
	by contradiction that%
	\[
	\varphi ^{3}(2^{\alpha }p)=\Phi ^{3}(2^{\alpha }p)
	\]
	following the same steps as Proof \ref{rem1} we get,
	
	\[
	\varphi ^{2}(p-1)=16\Phi ^{2}(p-1)
	\]
	then,
	
	\[
	\varphi ^{2}(p-1)>\Phi ^{2}(p-1)
	\]
	
	which contradicts Corollary \ref{gr}, then $n$ does not satisfy Eq. \ref{3}.
\end{proof}

We note that, by using similar arguments mentioned in the previous proofs, one
can prove the following,

	\begin{enumerate}
		\item Let $n=2^{\alpha }p^{2}$ where $\alpha \geq 1$. If $\varphi
		^{3}(n)=\Phi ^{3}(n),$ then, $p=3$ and $\alpha =1$.
		
		\item Let $n=2^{\alpha }p^{\beta }$ where $\alpha \geq 1$ and $\beta \geq 1$
		and $p$ is an odd prime, then $p=3$ and $\alpha =1$.
		
		\item If $n$ two or more odd prime factors then $\varphi ^{3}(n)\neq \Phi
		^{3}(n).$
		
		We summarize the Lemmas to formulate our main result of this section stated in the next theorem.
	\end{enumerate}

\begin{theorem}
	$\varphi ^{3}(n)=\Phi ^{3}(n)$ if and only if $n=5,10,$ $12,$ a divisor of $8
	$, $2\times 3^{a}$, or $2p$ where $p$ is a prime of the form $2q^{b}+1$
	where $q$ is an odd prime.
\end{theorem}

We note that the value of $n$ mentioned in the above theorem, are some of
the values that make $U^{2}(%
\mathbb{Z}
_{n})$ cyclic. In other words, if $\varphi ^{3}(n)=\Phi ^{3}(n)$, then $%
U^{2}(%
\mathbb{Z}
_{n})$ is cyclic (see \cite{11} ).

\end{document}